\pdfoutput=1
\documentclass[11pt,reqno]{amsart}
\usepackage[letterpaper,margin=1in,footskip=0.25in]{geometry}
\usepackage{mathrsfs}
\usepackage{amssymb}
\usepackage{mathtools}
\usepackage{tikz-cd}
\usepackage{enumitem}

\PassOptionsToPackage{pdfusetitle,pagebackref,colorlinks}{hyperref}
\usepackage{bookmark}
\hypersetup{
  linkcolor={red!50!black},
  citecolor={green!40!black},
  urlcolor={blue!50!black}
}
\newcommand{\onto}{{\hspace{.3em}\longrightarrow\hspace{-1.3em}\longrightarrow}\hspace{.3em}}
\newcommand{\M}{{\mathcal{M}}}

\newcommand{\FF}{{\mathcal{F}}}
\newcommand{\EE}{{\mathcal{E}}}

\newcommand{\Q}{{\mathbb Q}}
\newcommand{\Z}{{\mathbb Z}}

\newcommand{\J}{{\mathcal J}}

\newcommand{\OO}{{\mathcal{O}}}
\newcommand{\Dmod}{\mathcal{D}}
\newcommand{\PP}{\mathbb{P}}
\newcommand{\coloneqq}{:=}

\DeclareMathOperator{\mult}{mult}
\DeclareMathOperator{\Ker}{Ker}
\DeclareMathOperator{\dr}{DR}
\DeclareMathOperator{\Pic}{Pic}

\DeclareMathOperator{\rk}{rk}
\newcommand{\maxid}{{\mathfrak{m}}}

\newcommand{\hyp}{\mathbb{H}}

\newcommand{\red}{\text{red}}

\newtheorem{theorem}{Theorem}[section]
\newtheorem{lemma}[theorem]{Lemma}

\newtheorem{alphtheorem}{Theorem}

\newtheorem{variant}[theorem]{Variant}
\newtheorem{alphcorollary}[alphtheorem]{Corollary}
\theoremstyle{definition}
\newtheorem{definition}[theorem]{Definition}

\theoremstyle{definition}
\newtheorem{remark}[theorem]{Remark}
\newtheorem{remnot}[theorem]{Remark/Notation}

\newtheoremstyle{cited}{.5\baselineskip\@plus.2\baselineskip\@minus.2\baselineskip}{.5\baselineskip\@plus.2\baselineskip\@minus.2\baselineskip}{\itshape}{}{\bfseries}{\bfseries .}{5pt plus 1pt minus 1pt}{\thmname{#1}\thmnumber{~#2}\thmnote{ \normalfont#3}}
\theoremstyle{cited}

\newtheorem{citedlem}[theorem]{Lemma}

\newtheoremstyle{citeddef}{.5\baselineskip\@plus.2\baselineskip\@minus.2\baselineskip}{.5\baselineskip\@plus.2\baselineskip\@minus.2\baselineskip}{}{}{\bfseries}{\bfseries .}{5pt plus 1pt minus 1pt}{\thmname{#1}\thmnumber{~#2}\thmnote{ \normalfont#3}}
\theoremstyle{citeddef}

\begin{document}
\title[Vanishing for Hodge ideals on toric varieties]{Vanishing 
for Hodge ideals on toric varieties} 
\author{Yajnaseni Dutta}
\address{Department of Mathematics\\Northwestern University\\
Evanston, IL 60208-2730, USA}
\email{\href{mailto:ydutta@math.northwestern.edu}{ydutta@math.northwestern.edu}}
\urladdr{\url{http://www.math.northwestern.edu/~ydutta/}}

\makeatletter
  \hypersetup{
    pdfauthor=Yajnaseni Dutta,
    pdfsubject=\@subjclass,
    pdfkeywords=\@keywords
  }
\makeatother

\begin{abstract} 
In this article we construct a
Koszul-type resolution of 
the $p^{\text{th}}$ exterior power of the sheaf of holomorphic differential forms 
on smooth toric varieties and 
use this to prove a Nadel-type vanishing theorem for 
Hodge ideals associated to effective $\Q$-divisors on smooth projective 
toric varieties. 
This extends earlier results of  
Musta\c{t}\u{a} and Popa.
\end{abstract}
\maketitle
\section{Introduction}
In this paper we prove a Nadel-type vanishing statement for Hodge 
ideals on smooth projective toric varieties. This extends similar
vanishing statements in \cite{MP16} 
and \cite{MP18}.

In a series of papers, Musta\c{t}\u{a} and Popa 
thoroughly studied Hodge ideals with the goal of
 understanding singularities and Hodge theoretic properties of hypersurfaces in smooth varieties. 
 These ideals are indexed by non-negative integers, with the $0^{\text{th}}$ 
 being a multiplier ideal (see Remark \ref{rmk:mult} below). 
 In particular,
 a hypersurface $D$ in a 
smooth variety $X$ has log canonical singularities if and only if
$I_0(D)\simeq \OO_X$. 
For the sake of applications it is fundamental for 
such gadgets to satisfy vanishing theorems, 
just like multiplier ideals.

To formulate our statement, 
let $X$ be a smooth projective toric variety of dimension $n$, with torus invariant divisors $D_i$, $i=1,\cdots, d$. Let $D$ be an effective $\Q$-divisor on $X$
such that
there is an integer
$\ell$ and a section $s$ of some line bundle $M^{\otimes\ell}$
satisfying
$D = \frac{1}{\ell}H$ where $H=(s=0)$.
Then we have the following
\begin{definition}[Property \hypertarget{prop:pk}{$P_k(D)$}]
	
With the above assumption on $D$, we say that a line bundle 
$L$ on a smooth toric variety $X$, satisfies property $P_k(D)$ if
the divisors 
$L+D_{\red}-D$ and\\
$L+D_{\tau_1}+\dots+ D_{\tau_p}+D_{\red}-D$
for all $1\leq p\leq k$ and for all $\tau_i\in\{1,\cdots, d\}$ are ample.
 Moreover, when $D$ and $D_i$'s are ample, we also allow $L\sim_{\Q} D-D_{\red}$.  
\end{definition}
\begin{remnot}
 When $D$ is reduced,
 $D= D_{\red}$ and therefore Property $P_k(D)$
 does not depend on $D$. In this case, we will use the notation $P_k$ 
 to emphasise the lack of dependence on $D$.
\end{remnot}

\begin{alphtheorem}\label{main}
With the notation as above, for a fixed integer $k$ and a line bundle
$L$ on $X$  
satisfying property $P_k(D)$  
we have
\[H^i\Big(X, \omega_X\big((k+1)D_{\red}\big)\otimes L\otimes I_k(D)\Big) = 0\text{ for all } i>0.\]
\end{alphtheorem}

In \cite{MP16, MP18} Musta\c{t}\u{a} and Popa 
showed such vanishing statements in general, for smooth projective
varieties, as well as for projective spaces and for abelian varieties
with much weaker hypotheses (see  \cite[Theorem 25.3]{MP16} and \cite[Variant 12.5, Remark 12.4]{MP18}). However their general
statement required
additional hypotheses. For instance, 
on smooth projective toric varieties, they require
$D$ to be \textsl{reduced} $(k-1)$-\textsl{log canonical} 
(i.e. $I_p(D)\simeq \OO_X(D_{\red}-\lceil D\rceil)$ 
for all $p\leq k-1$) and  
the ample line bundle
$L$ to satisfy certain positivity properties with respect to
$D$ (see \cite[Corollary 25.1]{MP16}). Here we only need to assume certain positivity properties,
depending essentially only on the toric data, on the line bundle $L$.
Our statement recovers the statement
for projective spaces as we will see in Example \ref{ex:proj}. 
\begin{remark}
 The assumption  
 that $\ell D\sim M^{\otimes \ell}$ is not necessary to define the Hodge ideal, but it
 is to ensure the gluing of the local 
$\Dmod$-modules that are used to define the Hodge ideals for $\Q$-divisors (see Remark
 \ref{rmk:localglobal} for more details). Existence of 
 such global objects is crucial for the purpose of the
 proof we present. Note that such 
 assumptions can always be realised on a finite flat
 modification of $X$ and $D$.  
\end{remark}

\begin{remark}\label{rmk:mult}
 We know from \cite[Proposition 10.1]{MP16}
 (or in the $\Q$-divisor setting 
 from \cite[Proposition 9.1]{MP18})
 that $I_0(D) \simeq \mathcal{J}((1-\epsilon)D)$, the multiplier ideal sheaf of $(1-\epsilon)D$ for 
 $0<\epsilon \ll 1$. Note that when $k=0$, Theorem \ref{main}
 is the same as Nadel vanishing for this multiplier ideal sheaf.
\end{remark}
On a smooth toric variety $X$ of dimension $n$, we have the short exact sequence (see \cite[Theorem 8.1.6]{CLS11})
\begin{equation}\label{this}0\to \
	\Omega_X^1 \to 
\bigoplus_{i=1}^d\OO_X(-D_i) \to \OO_X^{\oplus d-n}\to 0\end{equation}
where the $D_i$'s are the torus invariant divisors of $X$.  Using this Euler type short exact sequence, we
 construct a resolution of the $k^{\text{th}}$ exterior power of the sheaf
 of holomorphic differential forms $\Omega^k_X$
on $X$
(see Lemma \ref{koz:toric}).
 Then an inductive argument,
similar to that in \cite[Theorem 25.3]{MP16} for projective spaces, yields the above vanishing statement.
We first look at some examples:
\subsection*{Examples}
\begin{enumerate}
	\item \textsl{Projective spaces:}\label{ex:proj} When $X=\PP^n$
	and $\deg D=d$, our statement recovers the statement in \cite[Theorem 15.3]{MP16}(Or, 
	\cite[Variant 12.5]{MP18}), namely
	\[H^i(X, \OO_X(\ell)\otimes\OO_X(kD_{\red})\otimes I_k(D)) = 0 \text{ for all } i>0\]
	and $\ell \geq d-n-1$. Indeed, torus invariant divisors on $\PP^n$ satisfy 
	$\OO_{\PP^n}(D_i) \simeq \OO_{\PP^n}(1)$ and therefore are ample. In the above statement taking $\OO_X(\ell)\simeq\omega_X\otimes L$,
	we see that when $\ell\geq d-n-1$, $L-D_{\red}$ satisfies 
	property \hyperlink{prop:pk}{$P_k(D)$}.
\item \textsl{Products of projective spaces:} When 
	$X=\PP^{n_1}\times\PP^{n_2}$,
and $D$ is
an effective $\Q$-divisor in the class $(c,d)$, we have 
\[H^i(X,\OO_X(a,b)\otimes \OO_X(kD_{\red})\otimes I_k(D))=0\text{ for all }
i>0\] whenever $a\geq c-n_1$
	and $b\geq d-n_2$.
	Indeed, the
torus invariant divisors are all of type $(1,0)$ or $(0,1)$ and therefore nef. 
 Then a line bundle
$L\simeq \OO_X(a',b')$ $L-D_{\red}$ satisfies property \hyperlink{prop:pk}{$P_k(D)$} if $a'\geq c+1$ and 
$b'\geq d+1$. It is worth mentioning that when $D$ is reduced, we get
\[H^i(X,\OO_X( (k+1)c-n_1, (k+1)d-n_2)\otimes I_k(D))=0 \text{ for all } i>0.\]
\item \textsl{Hirzebruch surfaces: }\label{ex:fr}
When $X=F_r=\PP(\OO\oplus\OO(r))$, the Hirzebruch surface with $r\geq 0$, we have $\Pic X = \Z F \oplus \Z E$, where $E\sim E'+rF$, $E'$ is the class of the section 
of self-intersection $-r$ and $F$ is the class of a fibre. It is well known that two of the four torus invariant divisors on
$X$ lie in the class of $F$ and the remaining 
two are in the classes of $E$ and $E'$. Note that
all torus invariant divisors but $E'$ are nef. Therefore to satisfy
property $P_k$, a line bundle $L \simeq \OO_X(aF + bE)$
must satisfy that $a\geq kr$ and $b\geq 1$.
Indeed, big and nef (resp. ample) bundles on $X$ are 
given by 
 $a,b\geq 0$ (resp. $>0$) (see \cite[Theorem 2.4(3)]{BC13}
or \cite[Example 6.1.16]{CLS11}) and
property \hyperlink{prop:pk}{$P_k$} above 
mandates $L+kE'$ to be ample. Then for a reduced curve $D$ in the class $(c,d)$ we have,
\[H^i(X, \OO_X((k+1)(c+r)-2,(k+1)d-1))=0 \text{ for all }i>0.\]
\end{enumerate}
\textbf{Applications.}
As an application we 
address the classical problem of finding the
number of conditions imposed on certain spaces of hypersurfaces by the isolated singular points
on a given singular hypersurface. 
\begin{alphcorollary}\label{thm:1jet}
 Let $D$ be a reduced effective divisor on a smooth projective toric variety $X$ of dimension $n$. 
 Let $S_m(D)$ denote the set of all isolated singular points in $D$ with multiplicity at least $m$. Then
 $S_m(D)$ imposes independent conditions on 
 the hypersurfaces in
 \[H^0\left(X, \OO_X\Big(\Big(\left\lfloor \frac{n}{m}\right\rfloor+1\Big) D - \sum_iD_i\Big)\otimes L\right)\]
 for 
any line bundle $L$ satisfying property \hyperlink{prop:pk}{$P_{\lfloor\frac{n}{m}\rfloor}$}.
 
\end{alphcorollary}
In particular, when $X=\PP^n$ we recover \cite[Corollary H]{MP16}. We discuss 
some examples in \S \ref{ex:applications}.

We can extend this analysis to the study of $(j-1)$-jets along $S_m(D)$,
just as was done in \cite[Corollary 27.3]{MP16}.
Recall that a line bundle $M$ is said to separate 
$(j-1)$-jets along a set of points $S\subset X$ if the map 
\[H^0(X,M)\onto \bigoplus_{p\in S}\OO_X/\maxid_p^j\]
is surjective. In particular, $S$ imposes independent conditions on
$H^0(X,M)$ means that $M$ separates 0-jets along $S$. For $m\geq 3$, define

 \begin{equation}
    k_{m,j} =
    \begin{cases*}
      \left\lceil \frac{j+n-m}{m}\right\rceil & $\text{ if }  j\leq m-1$\\
      \left\lceil \frac{j+n-m}{m-2}\right\rceil &$ \text{ if }  j\geq m$\\
    \end{cases*}
  \end{equation}
	Then we obtain the following:
\begin{alphcorollary}\label{cor:jets}
With the assumptions and notation in
Corollary \ref{thm:1jet}, for $m\geq 3$ 
the space of hypersurfaces
\[H^0\Big(\OO_X\big(\big(k_{m,j}+1\big)D - \sum_iD_i\big)\otimes L\Big)\] separate	 $(j-1)$-jets along $S_m(D)$ for 
any line bundle $L$ satisfying the property \hyperlink{prop:pk}{$P_{k_{m,j}}$}.
\end{alphcorollary}
\subsection*{Acknowledgements}
I am greatly thankful to my advisor Mihnea Popa for
suggesting the problem and for his constant support
and encouragement
throughout the preparation of this paper. I would also like to 
thank Mircea Musta\c{t}\u{a} for various discussions regarding toric
varieties and Emanuel Reinecke for various comments on an earlier version of this draft. Lastly, I would like to thank both Mircea Musta\c{t}\u{a} and Mihnea Popa for spurring my interest in this area with their 
 series of articles on Hodge ideals.  
\section{Preliminaries}
\subsection{Eagon-Northcott complexes}
We first discuss the general theory behind 
the construction of a Koszul-type complex that
will be used
in the proof of Theorem \ref{main}. The main reference for this section is \cite[Appendix B]{Laz04}.
The following lemma is surely well-known to experts, nonetheless we include the proof for completeness.\\

Consider a short exact sequence of vector bundles on a variety $X$:
\begin{equation}\label{basic}0\to \Omega\to \EE \to \FF \to 0\end{equation}
such that $\rk \EE = e$, $\rk \FF = f$ with $e>f$. 
Hence, $\rk \Omega = e-f$. We denote $\bigwedge^p\Omega =: \Omega^p$.  
We then have the following resolution
of $\Omega^p$ in terms of $\EE$ and $\FF$ (see also 
[\textit{loc.\ cit.} Appendix B.2]):
\begin{lemma}\label{ENp}
The following complex is exact and gives a resolution for $\Omega^p$:
\begin{equation*}
	0\to\bigwedge^e\EE\otimes S^{e-p-f}\FF^{\vee}\otimes\big(\bigwedge^f\FF\big)^{\vee}\to \dots\to \bigwedge^{p+f+j}\EE\otimes S^j\FF^{\vee}\otimes\big(\bigwedge^f\FF\big)^{\vee} 
\to \dots\to\bigwedge^{p+f}\EE\otimes\big(\bigwedge^f\FF\big)^{\vee}
\end{equation*}
where $S^j\FF^{\vee}$ denotes the $j^{\text{th}}$ symmetric
power of $\FF^{\vee}$.
\end{lemma}
\begin{proof}
Since the morphism $\EE\to\FF$ is a map of vector bundles, 
we obtain by \cite[Theorem B.2.2.]{Laz04}, that the 
$p^{\text{th}}$ 
Eagon-Northcott complex ($EN_p$)
\begin{multline*}
0\to\bigwedge^e\EE\otimes S^{e-p-f}\FF^{\vee}\otimes\big(\bigwedge^f\FF\big)^{\vee}\to \dots\to \bigwedge^{p+f+j}\EE\otimes S^j\FF^{\vee}\otimes\big(\bigwedge^f\FF\big)^{\vee} \to\cdots\to \bigwedge^{p+f}\EE\otimes \big(\bigwedge^f\FF\big)^{\vee}\to\\
\to \bigwedge^p\EE\xrightarrow{\phi} \bigwedge^{p-1}\EE\otimes \FF\to \cdots\to E\otimes S^{p-1}\FF\to S^p\FF\to 0
\end{multline*}
is exact. 

To determine the kernel of $\phi\colon\bigwedge^p\EE\xrightarrow{} \bigwedge^{p-1}\EE\otimes \FF$ above,
we need to analyse the construction of $(EN_p)$.
Consider the projective
bundle
$\PP(\FF)\coloneqq \underline{\text{Proj}}\big(\bigoplus_iS^i\FF\big)$
with the map $\pi\colon \PP(\FF) \to X$
and the map of vector bundles
\[\pi^*\EE(-1)\to \OO_{\PP(\FF)}.\]
Denote by $K:= \Ker\big(\pi^*\EE(-1)\to \OO_{\PP(\FF)}\big)$ the kernel vector bundle. Then 
\[\Ker(\phi) = \Ker \Big(\pi_*\bigwedge^p\pi^*\EE \xrightarrow{\phi} \pi_*\bigwedge^{p-1}\pi^*\EE(1)\Big) \simeq \pi_*(\bigwedge^pK)(p).\]
But the last term is isomorphic to $\Omega^p$.
Indeed,
snake lemma applied to the following exact grid:
\begin{equation*}
\begin{tikzcd}
&&0\arrow{d}&0\arrow{d}&\\
&&\pi^*\Omega(-1) \arrow{d}& K\arrow{d}\\
&0\arrow{r}& \pi^*\EE(-1) \ar[equal]{r}\arrow{d}&\pi^*\EE(-1)\arrow{d}\arrow{r} & 0\\
0 \arrow{r} & \Omega^1_{\PP(\FF)/X}\arrow{r}\ar[equal]{d} 
& \pi^*\FF(-1) \arrow{d}\arrow{r}&\OO_{\PP(\FF)}\arrow{d}\arrow{r} &0 \\
&\Omega^1_{\PP(\FF)/X} &0&0&
\end{tikzcd}
\end{equation*}
gives us the short exact sequence
\[0\to \pi^*\Omega\to K(1)\to \Omega_{\PP(\FF)/X}(1)\to 0.\]
(In the above diagram, the horizontal short exact sequence is the relative Euler sequence and the middle vertical sequence is the pullback of (\ref{basic}) under a smooth morphism.)\\
Now for a fixed $p$ we have a filtration 
\[(\bigwedge^pK)(p)=F_0\supseteq F_1\supseteq \dots\supseteq F_{p+1}= 0\]
such that \[F_k/F_{k+1} = \pi^*\Omega^k \otimes \Omega_{\PP(\FF)/X}^{p-k}(p-k).\]
We assert that \[\pi_*\Omega_{\PP(\FF)/X}^{p-k}(p-k)=0 \text{ for all }p-k\geq 0.\] Granted this, 
$\pi_*F_{k+1} \simeq \pi_* F_{k}$
for all $k\leq p$. In particular, $\Omega^p \simeq \pi_*F_p\simeq \pi_*F_0 \simeq \pi_*(\bigwedge^pK(p)).$

The assertion follows from the Koszul resolution of 
$\pi_*\Omega_{\PP(\FF)/X}^{k}(k)$, namely:
\[\bigwedge^{f}\pi^*\FF(-k+f)\to\cdots\to\bigwedge^{k+1}\pi^*\FF(-1) \to \Omega_{\PP(\FF)/X}^{k}(k)\to 0\]
and the equality
$R^i\pi_*\bigwedge^{k+j}\pi^*\FF(-j) = 0$ for all $i$ and for 
all $1\leq j\leq f-1$ (see \cite[Exer. III.8.4.(c)]{Har77}). 

\end{proof}

We apply the above lemma to smooth toric varieties to obtain
a resolution of $\Omega_X^p$.
\begin{lemma}\label{koz:toric}
Let $X$ be a smooth toric variety with
torus invariant divisors $D_i$, $i=1,...,d$ then, $\Omega_X^{p}$ admits the following
Koszul-type resolution:
\begin{multline*}
0\to \bigoplus\omega_X\to\dots\to\bigoplus\bigoplus_{\sigma\in S_{n-p-j}}\omega_X\big(D_{\sigma_1}+\dots+D_{\sigma_{n-p-j}}\big)\to\\
\dots\to \bigoplus \bigoplus_{\sigma\in S_{n-p}}\omega_X\big(D_{\sigma_1}+\dots+ D_{\sigma_{n-p}}\big)\to \Omega_X^{p}\to 0.
\end{multline*}
Here $S_j$ is the set of all ordered sequences of length $j$ in $\{1,\dots,d\}$.
\end{lemma}
\begin{proof}
Lemma \ref{ENp} applied to the short exact sequence (\ref{this}) gives the following resolution:
\begin{multline*}
0\to \bigoplus\bigwedge^{d}\Big( \bigoplus_{i=1}^d\OO_X\big(-D_i\big)\Big)\to
\dots\to\bigoplus\bigwedge^{d-n+p+j}\Big(\bigoplus_{i=1}^d\OO_X\big(-D_i\big)\Big)\to \\
\cdots\to\bigoplus\bigwedge^{d-n+p}\Big(\bigoplus_{i=1}^d\OO_X\big(-D_i\big)\Big)\to \Omega_X^{p}\to 0.
\end{multline*}
Since $\omega_X\simeq \OO_X(-\sum_{i=0}^dD_i)$ \cite[Theorem 8.2.3.]{CLS11},
we can rewrite each term in the above long exact sequence as the 
corresponding one from the 
long exact sequence in the statement.
\end{proof}
\subsection{Preliminaries on Hodge ideals}
Let $D$ be a reduced effective divisor on a smooth variety $X$. Hodge ideals associated to $D$ arise as a measure of the deficit between the Hodge filtration and the pole-order filtration on the Hodge module
$\OO_X(*D)\coloneqq \bigcup_{k\geq 0}\OO_X(kD)$. If $D$ is
smooth, one has $F_{k}\OO_X(*D) = \OO_X((k+1)D)$, in other words $I_k(D) = \OO_X$. 

Similarly, when $D$ is an effective $\Q$-divisor
and $H$ is an integral divisor such that $D = (1-\beta)H$
with $\beta<1$, one
can define Hodge ideals associated to $D$. However
in this case one needs to consider the rank 1 free $\OO_X(*D_{\red})$-module generated by $h^{\beta}$, where $h$ is a local equation of $H$.
 It turns out that these local $\Dmod$-modules appear 
 as direct summand of certain Hodge modules, namely the pushforward of 
the sheaf of meromorphic functions with poles along 
the preimage of $D_{\red}$, under some local 
cyclic cover along $h$ (see \cite[Lemma 2.6]{MP18}). Therefore one can
 make sense of a 
 filtration on $\OO_X(*D_{\red})h^{\beta}$ induced
from the Hodge filtration of the ambient Hodge module. 
Moreover this filtration coincides with the Hodge filtration on $\OO_X(*D)$ whenever
$D$ is reduced. We refer the reader to 
\cite{MP18} for details.

\begin{definition}
	Let $D$, $H$ and $h$ be as before. Then
	for an integer $k\geq 0$, the $k^{\text{th}}$-Hodge ideal
	associated to $D$ is defined by the 
	following equation:
  \[F_k\OO_X(*D_{\red})h^{\beta} =
  I_k(D)\otimes_{\OO_X} \OO_X(kD_{\red}+H)h^{\beta}.\]
  The filtration $F$ coincides with the Hodge filtration 
	when $D$ is reduced.
\end{definition}
\begin{remark}\label{rmk:localglobal}
The local $\Dmod$-modules $\OO_X(*D_{\red})h^{\beta}$ do not, in general,
glue to a global object. However by \cite[Remark 2.14]{MP18}, $I_k(D)$ is 
independent of the choice of $h$ and $\beta$ and hence is defined globally. 
The additional assumptions on $D$, in the statement of Theorem \ref{main}, ensure a global object. In other words, when we have
$\beta = \frac{\ell - 1}{\ell}$
and $\OO_X(H)\simeq M^{\otimes \ell}$ for some 
integer $\ell$ and some
line bundle $M$ on $X$, the local
$\Dmod$-modules $\OO_X(*D)h^{\beta}$ glue.
	Following 
	\cite[\S 5]{MP18}, we denote this global $\Dmod$-module by
	$\M_1$. As an $\OO_X$-module, we have
	\[\M_1\simeq M\otimes_{\OO_X}\OO_X(*D_{\red}).\]
	For the proof of Theorem \ref{main} that we present, it is crucial
	that we work with a globally defined $\Dmod$-module. 
\end{remark}

The induced Hodge filtration on $\M_1$ is given by (see, for instance,
the proof of \cite[Theorem 12.1]{MP18}) 
	\[F_k\M_1\simeq M(-D_{\red})\otimes_{\OO_X}\OO_X((k+1)D_{\red})\otimes I_k(D)\]
to which one can associate its
filtered deRham complex $\dr(\M_1)$.
	\begin{definition}[The de-Rham complex]
For any non-negative integer $k$, the $k^{\text{th}}$ filtered piece of the de-Rham complex of $\M_1$ is defined by 
\[
F_k\dr(\M_1) = [F_k\M_1\to \Omega_X^1 \otimes F_{k+1}\M_1\to\cdots\to \Omega^n_X\otimes F_{k+n}\M_1]
\]
and the associated graded complex is given by
\[
gr^F_k\dr(\M_1) = [gr^F_k\M_1\to \Omega_X^1 \otimes gr^F_{k+1}\M_1\to\cdots\to \Omega^n_X\otimes gr^F_{k+n}\M_1]
\]
as complexes in degrees $-n,\dots,0$.
\end{definition}
The higher hypercohomologies of the associated graded complex 
of Hodge modules satisfy a vanishing statement due to Saito 
(see \cite[\S 2.g]{Sai90}, also see \cite{Sch14}, \cite{Pop16}). 
As a result, one also obtains vanishing statements for the
associated graded complex of the filtered direct summand $\M_1$. We therefore have:
\begin{citedlem}[(Saito vanishing)]\label{thm:svanish}
Let $A$ be an ample line bundle on $X$. Then,
\[\hyp^i(X, gr^F_k\dr(\M_1)\otimes A) = 0 \text{ for all } i>0, k\geq 0\]
\end{citedlem}
The following vanishing is a consequence of Artin vanishing (see \cite[Proposition 22.1]{MP16} for instance):
\begin{lemma}\label{thm:mpvanish}
Let $X$ be a smooth projective variety and $D$ an effective divisor on $X$ such that $X\setminus D_{\red}$ is affine (when $D$ is ample for instance), then
\[\hyp^i(X, gr^F_k\dr(\M_1)) = 0 \text{ for all } i>0, k\geq 0.\] 
\end{lemma}

\section{Proof of the vanishing theorem}
We will prove the following: 
 \setcounter{alphtheorem}{0}
\begin{variant}[Variant of Theorem \ref{main}\label{thm:main}]
With the hypotheses of Theorem \ref{main}, we have for all 
integers $k \geq 0$ the following equivalent statements hold.
\begin{enumerate}
\item $H^i\big(X,\omega_X((k + 1)D_{\red}) \otimes L\otimes I_k(D)\big) = 0$ for all $i > 0$.
\item $H^i\Big(X,\omega_X\otimes gr^F_k \M_1\otimes A\Big)= 0$ for all $i > 0$
where $A\otimes M(-D_{\red})=L$. 
\end{enumerate}
\end{variant}

The proof uses the vanishing
Lemmas \ref{thm:svanish} and \ref{thm:mpvanish}, 
followed by an inductive argument
involving a spectral sequence. The base 
case of the induction is Nadel vanishing for $I_0(D)$, 
which by Remark \ref{rmk:mult}, is the multiplier ideal
$\J((1-\epsilon)D)$. The argument is, in spirit, very similar to
the proof for projective spaces as in \cite[Theorem 25.3]{MP16}.

\begin{proof}
The filtration on $\M_1$ is given by
$F_k\M_1 \simeq \OO_X((k+1)D_{\red})\otimes M(-D_{\red}) \otimes I_k(D)$
and therefore we have a short exact sequence:
\[0\to \omega_X(kD_{\red})\otimes L \otimes I_{k-1}(D) \to \omega_X((k+1)D_{\red})\otimes L \otimes I_k(D) \to \omega_X\otimes gr^F_k\M_1\otimes A\to 0\] where $A\coloneqq L (D_{\red}-D)$.
Since by hypothesis $A$ is ample and $D\sim_{\Q} M$, 
 Nadel vanishing theorem gives,  
 \[H^i(X, \omega_X(D_{\red})\otimes L \otimes I_0(D))=0 \text{ for all }i>0.\] Therefore, by induction the above statements are equivalent.

We aim to show the second statement. 
We have, \begin{equation}
	gr_{k}^F\dr(\M_1)[-n]= [gr^F_k\M_1\to\Omega_X^{1}\otimes gr^F_{k-n+1}\M_1 \to\dots\to\Omega^{n-1}\otimes gr^F_{-1+k}\M_1 \to \omega_X\otimes gr^F_{k}\M_1]
\end{equation}
in degrees $0$ to $n$.

Consider the spectral sequence associated to the complex $C^{\bullet}:=  gr_{k}^F\dr(\M_1)[-n]\otimes A$:
\[E^{p,q}_1:= H^q(X, C^p) \Longrightarrow \mathbb{H}^{p+q}(X, {C}^{\bullet}).\]
Now by Saito vanishing (Lemma \ref{thm:svanish}), we have
$\mathbb{H}^{p+q}(X, {C}^{\bullet}) = 0$ when $p+q \geq n+1$.
Also note that, because of the length of $C^{\bullet}$, the spectral
sequence degenerates at level $n+1$, therefore, $E_{n+1}^{p,q} 
\simeq E_{\infty}^{p,q}$. The latter, $E_{\infty}^{p,q}$, a quotient of $\mathbb{H}^{p+q}(X, {C}^{\bullet})$, also vanishes
for $p+q \geq n+1$. 
As the vanishing statement that we aim for, is $E_1^{n,i}=0$ for all $i>0$. 
It is thus enough to
show that $E_1^{n,i} \simeq \dots \simeq E_{n+1}^{n,i}$.
In other words, it
is enough to show that, $E_r^{n-r,i+r-1}=0$ for all $r$ 
with $1\leq r\leq n$.
Indeed, for each $r\geq 1$, the differential going out of 
$E^{n,i}_r$ maps to $E_r^{n+r, i-r+1}$, which is 0 because
of the length of $C^{\bullet}$. In fact we will show that,
\[E_1^{n-r,i+r-1}= H^{i+r-1}(X, \Omega^{n-r}_X\otimes A \otimes gr_{k-r}^F\M_1)=0.\] 

By Lemma \ref{koz:toric}, $\Omega_X^{n-r}$ has the resolution:
\begin{multline*}
0\to \bigoplus\omega_X\to\dots\to\bigoplus\bigoplus_{\sigma\in S_{r-j}}\omega_X\otimes \OO_X\big(D_{\sigma_1}+\dots+D_{\sigma_{r-j}}\big)\\
\to\dots\to \bigoplus \bigoplus_{\sigma\in S_r}\omega_X\otimes \OO_X\big(D_{\sigma_1}+\dots+ D_{\sigma_r}\big)\to \Omega_X^{n-r}\to 0
\end{multline*}
where $S_j$ is the set of all ordered sequences of length $j$ in $\{1,\dots,d\}$. Tensoring this long exact sequence with $gr^F_{k-r}\M_1\otimes A$, 
we see that it is enough to show,
\[H^{i+r-1}\big(X, \omega_X\otimes A(D_{\sigma_1}+\dots+D_{\sigma_j})\otimes gr^F_{k-r}\M_1\big)=0\]
for all $j\leq r$. But this follows by induction, since $L'\coloneqq L(D_{\sigma_1}+\dots+D_{\sigma_j})$ with $j\leq r$ satisfies 
property \hyperlink{prop:pk}{$P_{k-r}(D)$} by our hypotheses. 

Further if $D$ and the $D_i$'s are ample, 
we can take $A\simeq \OO_X$. In this case, to begin the induction,
we resort to the Artin-type vanishing in Lemma \ref{thm:mpvanish} and 
obtain:
$$\mathbb{H}^{p+q}\big(X, C^{\bullet}\big)=0
\text{ for all } p+q\geq n+1.$$
The latter steps of the induction use Saito vanishing (Lemma \ref{thm:svanish}) with
$A'\coloneqq \OO_X(D_{\sigma_1}+\dots+D_{\sigma_j})$, which are automatically
ample due to the ampleness assumption on the $D_i$'s. The remaining argument follows
the first part of the proof verbatim.
\end{proof}

\section{Applications}
In this section the divisor
$D$ is reduced and effective.

We now apply Theorem \ref{main} to deduce Corollaries \ref{thm:1jet} and \ref{cor:jets}. In \cite[Corollary 27.2]{MP16} similar results for projective spaces were deduced.
The key point is to understand how deep the Hodge ideals sit
inside the maximal ideal of an isolated singular point.
This has already been studied in 
 [Theorem E, Corollary 19.4 \emph{loc.\ cit.}]. Without
proof we collect their results in the following:
\begin{lemma}\label{lem:mult}
Let $D$ be a reduced hypersurface on a smooth variety $X$ of dimension $n$. Let $x$ be an isolated singular point of $D$
with $\mult_x(D) =m$. Let $k$ be an integer such that 
$\frac{n}{k+1}<m< \frac{n}{k}$, then 
\[I_k(D) \subseteq \maxid_{x}^{(k+1)m-n}.\]
Further if $m \geq  \frac{n}{k}$, then 
\[I_k(D) \subseteq \maxid_x^{\ell} \quad\text{ where } \ell=\max\big\{m-1, (k+1)(m-2)-n-2\big\}\]
\end{lemma}

\begin{proof}[Proof of Corollary \ref{thm:1jet}]
Recall that $S_m(D)$ is the set of all isolated singular points on $D$ with multiplicity at least $m$.
	From Lemma \ref{lem:mult} we know that for $k=\left\lfloor{\frac{n}{m}}\right\rfloor$ and for all $p\in S_m(D)$, $I_k(D)_p\subseteq \maxid_p$.
 Therefore for a line bundle $L$ satisfying 
 \hyperlink{prop:pk}{$P_k$}, we have:
 \[H^1\Big(X, \OO_X\Big(\Big(\left\lfloor {\frac{n}{m}}\right\rfloor+1\Big)D - \sum_iD_i\Big)\otimes L \otimes I_k(D)\Big) = 0.\]
    This yields the following surjection
   \[H^0\Big(X, \OO_X\Big(\Big(\left\lfloor \frac{n}{m}\right\rfloor+1\Big)D - \sum_iD_i\Big)\otimes L\Big) \onto \bigoplus_{p\in S_m} \OO_X/\maxid_p\]
	 thereby proving the statement.
\end{proof}
Note that the choice of $k$ is optimal for this proof. For example, by \cite[Theorem D]{MP16}, when $p$ is an ordinary
isolated singular point of $D$ and 
$k\leq \lfloor\frac{n}{m}\rfloor -1$, locally around $p$ we have 
\[I_k(D)_p\simeq \OO_{X,p}.\]\\
The proof of separation of $(j-1)$-jets follows similarly. 
For the sake of completeness, we include it here.

\begin{proof}[Proof of Corollary \ref{cor:jets}]
	First note that $I_{k_{m,j}}(D) \subseteq \maxid_p^j$ for all $p\in S_m(D)$. Indeed, if $m'>m$, then $k_{m,j}=k_{m',j'}$ for $j'\geq j$
	and therefore, if $\mult_{p'}D=m'>m$ for some point $p'\in S_m(D)$, then $I_{k_{m,j}}(D) \subseteq \maxid_{p'}^{j'}\subseteq \maxid_{p'}^{j}$.
	Now for a line bundle $L$ satisfying $P_{k_{m,j}}$, we have the vanishing in Theorem \ref{main}
	for $I_{k_{m,j}}(D)$ and therefore
    we obtain a surjection:
    \[H^0\big(X, \OO_X\big((k_{m,j}+1)D - \sum_iD_i\big)\otimes L\big) \onto \bigoplus_{p\in S_m} \OO_X/\maxid_p^j.\]
		Hence the Corollary.
\end{proof}
\subsection*{Examples.} \label{ex:applications}
\begin{enumerate}
 \item 
 Let $X=\PP^{n_1}\times \PP^{n_2}$ and denote $n\coloneqq n_1+n_2$. Let $D$ be a reduced effective divisor 
of type $(c,d)$. 
Since the toric divisors on $X$ are 
nef, 
any ample line bundle, and in particular $L=\OO_X(1,1)$,
satisfies property \hyperlink{prop:pk}{$P_k$} 
for for all $k$. Therefore, $S_m(D)$
imposes independent conditions on the space of hypersurfaces 
\[H^0\left(X,\left(\Big(\left\lfloor\frac{n}{m}\right\rfloor+1\Big)c-n_1, 
\Big(\left\lfloor\frac{n}{m}\right\rfloor+1\Big)d-
n_2\right)\right).\] 
In particular, isolated singular points on a surface in the class $(c,d)$ on $X=\PP^2\times \PP^1$ 
impose conditions on $|\OO_X(2c-2, 2d-1)|$. This can be compared with the Severi-type 
bound as in \cite[Main Theorem]{PW06} for $\PP^3$, namely isolated singular points on a surface of degree $d$
in $\PP^3$ impose independent conditions on hypersurfaces of degree at least $2d-5$. \\
\item
Let $X=F_r$, the Hirzebruch surface with $r\geq 0$, as in
Example \ref{ex:fr} from the introduction. Note that the line bundle
 $L= \OO_X((rF+E))$ 
 satisfies \hyperlink{prop:pk}{$P_1$}. Moreover, we have $K_X\sim(r-2)F-2E$. Then for
any reduced effective singular divisor
$D$ in the class $(c,d)$, $S_2(D)$
imposes independent conditions on the 
space of hypersurfaces
\[H^0(X,\OO_X((2(c+r)-2)F+(2d-1)E)).\]

\end{enumerate}

\end{document}